\documentclass[4paper,leqno,[11pt]{amsart}

\usepackage[top=1.5in,bottom=1.3in,left=1.3in,right=1.3in,marginparwidth=1.5cm]{geometry}

\usepackage[utf8]{inputenc}

\usepackage{amsmath,amsthm,amssymb,amsfonts}

\usepackage[backref=page]{hyperref}
\renewcommand*{\backref}[1]{}
\renewcommand*{\backrefalt}[4]{%
    \ifcase #1 (Not cited.)%
    \or        (Cited on page~#2.)%
    \else      (Cited on pages~#2.)%
    \fi}

\hypersetup{
	colorlinks   =true,
	citecolor    =[rgb]{0.3,0,0.3},
	linkcolor    =blue,
	urlcolor     =[rgb]{0.3,0,0.3}
}

\numberwithin{equation}{section}

\usepackage{tikz}
\usetikzlibrary{cd}

\newtheorem{thm}{Theorem}[section]
\newtheorem{lem}[thm]{Lemma}
\newtheorem{prop}[thm]{Proposition}
\newtheorem{cor}[thm]{Corollary}

\newtheorem{defn}[thm]{Definition}
\newtheorem{exmp}[thm]{Example}
\newtheorem{qsn}[thm]{Question}
\newtheorem{rem}[thm]{Remark}
\newtheorem*{thankyou}{Acknowledgements}
\newtheorem*{thmA}{Theorem A}
\newtheorem*{thmB}{Theorem B}
\newtheorem*{thmC}{Theorem C}
\newtheorem*{thmD}{Theorem D}

\theoremstyle{remark}

\newtheorem*{notation}{Notation}

\renewcommand{\P}{\mathbb{P}}

\newcommand{\Z}{\mathbb{Z}}
\newcommand{\R}{\mathbb{R}}
\newcommand{\C}{\mathbb{C}}

\renewcommand{\O}{\mathcal{O}}
\newcommand{\Nef}{\mathrm{Nef}}
\newcommand{\Pic}{\mathrm{Pic}}
\newcommand{\F}{\mathcal{F}}
\newcommand{\KC}{\mathcal{K}}

\title{Higher Gaussian Maps on K3 surfaces.}
\author{Angel David Rios Ortiz}
\address{Sapienza Universita di Roma, Dipartimento di Matematica, Piazzale Aldo Moro 5, 00185 Roma}
\email{rios@mat.uniroma1.it}
\address{Max-Planck-Institut f\"{u}r Mathematik in den Naturwissenschaften, Inselstrasse 22, 04103 Leipzig, DE.}
\email{arios@mis.mpg.ide}

\begin{document}

\maketitle

\begin{abstract}
  We give sufficient conditions for the surjectivity of higher Gaussian maps on a polarized K3 surface. As an application, we show that the $k$-th Gaussian map for a general curve of genus $g$ (that depends quadratically with $k$) is surjective. Along the proof we also exhibit an ampleness criterion for divisors in the Hilbert scheme of two points of a K3 surface.
\end{abstract}

\section{Introduction}

For a smooth curve $C$, the so-called first Gaussian (or Wahl) map is the morphism
\begin{equation}\label{eq:wahlmap}
\gamma^1_C : H^0(C, \omega_C) \wedge H^0(C, \omega_C) \to H^0(C, \omega_C^{\otimes 3})
\end{equation}
defined essentially by $\gamma^1_C(s\wedge t) = s\mathrm{d}t - t\mathrm{d}s$. This map has been studied intensively by several authors because it is related to the deformation and extendability properties of the curve in its canonical embedding. The striking properties of this map were first put in evidence by a theorem of Wahl \cite{Wahl87}, saying that if $C$ lies as a hyperplane section of some K3 surface, then $\gamma^1_C$ is never surjective. This last fact contrasts with \cite{CHM88}, where the authors prove that for the \emph{general} curve of genus $g\geq 10$ (by a dimension count this is the smallest possible value of $g$) and $g\neq 11$ (the generic curve of genus $11$ lies on a K3 surface) the Wahl map \emph{is} surjective (see also \cite[Theorem 1]{Voisin92} for a different proof). Moreover, a recent result due to Arbarello, Bruno and Sernesi \cite{ABS17} actually characterizes Brill-Noether-Petri general curves that lie as hyperplane sections on a K3 surface, or on a limit thereof, as \emph{exactly} those where the Wahl map is not surjective.

The Wahl map can be defined for any variety and every line bundle; and can be seen as the first instance of a hierarchy of maps called the \emph{higher Gaussian maps} of a variety, see \cite{Wahl92} or Section \ref{section:gaussianandhilbertscheme} for the relevant definitions. In particular, the second Gaussian map $\gamma^2_C$ has been studied extensively \cite{CF09},\cite{ColFred10}, \cite{CPT01} for its relation with the curvature of $M_g$ in $A_g$. It was proved in \cite{CCM11} that for \emph{general} curves of genus $g\geq 18$ (this is the minimum possible) the second Gaussian map is surjective. In \cite{CCM11} the authors ask about the rank of the higher Gaussian maps for a general curve in $M_g$. As one of the main results of this paper, we will give a partial (and asymptotically optimal) answer to the above question.

In this work, we will relate the surjectivity of higher Gaussian maps of a regular surface $S$ equiped with a line bundle $L$ with the cohomology of linear systems $L-(k+2)\delta$ on the Hilbert scheme of two points on $S$. Here we are also denoting by $L$ the induced line bundle on the Hilbert scheme (see Section \ref{sec:linearsystemsonthehilbertscheme} for the notations in the theorem). Our first result is the following.
\begin{thmA}\label{thm:thmA}
Let $S$ be a projective surface with $h^1(S,\O_S) = 0$ and let $L$ be a line bundle on $S$. If $H^1(S^{[2]},L-(k+2)\delta) = 0$ then $\gamma^k_L$ is surjective.
\end{thmA}

Let $S$ be now a K3 surface. Building on the work of Bayer-Macrì on the birational geometry of $S^{[2]}$, we give a criterion for ampleness for divisors on $S^{[2]}$ that depends only on the geometry of $S$. Essentially a class will be ample if there are no rational nor elliptic curves of \emph{low degree} lying on $S$.

\begin{thmB}
Let $(S,L)$ be a polarized K3 surface with $L^2 = 2d$. Suppose $d>\frac{4a^2}{3}$, then the class $L-a\delta$ is movable if and only if there are no effective divisors $D$ in $S$ such that $D^2 = 0$ and $L\cdot D\leq 2a$. If $d>\frac{9a^2}{4}$, then the class $L-a\delta$ is ample if and only if it is movable and there are no effective divisors $D$ in $S$ such that $D^2 = -2$ and $L\cdot D < a$.
\end{thmB}

The Hilbert scheme of two points on a K3 surface is also an example of a \emph{hyperk\"ahler manifold}. Hence the knowledge of several vanishing theorems for line bundles on hyperk\"ahler varieties along with the structure of the nef cone let us prove surjectivity of all higher Gaussian maps for general K3 surfaces. 

\begin{thmC}
Let $(S,L)$ be a polarized K3 surface of degree $2d$ with $\Pic(S) = \Z L$, then $\gamma^k_L$ is surjective for all $k>0$.
\end{thmC}

Hyperplane sections of polarized $K3$ surfaces are canonical curves. Our main application of the previous theorem gives surjectivity of higher Gaussian maps of canonical curves, thus solving the problem raised in \cite{CCM11} for high enough genus.

\begin{thmD}
Let $k>1$ be an integer. Then for a general curve of genus $g> 4(k+2)^2 + 2$ the $k$-th higher Gaussian map is surjective.
\end{thmD}

Observe that although the bound is not optimal, it is very close to it, in the sense that is quadratic in $k$.

This paper is organized as follows: In Section \ref{section:gaussianandhilbertscheme} after some preliminaries we will study the relation between higher Gauss maps and linear series on the Hilbert Scheme of $2$ points on a regular surface, proving Theorem A in \ref{thm:vanish}. Since we will use some terminology on equivariant sheaves, for completeness we include the relevant definitions in the Appendix \ref{section:equivariantsheaves}.

Starting from Section \ref{sec:amplenesscriterion} we specialize to K3 surfaces. In all the rest of the paper the celebrated results of Bayer and Macrì \cite{BMMMP14} will play a central role, along with hyperk\"ahler geometry. We will prove in Corollaries \ref{cor:movablehighdegree} and \ref{cor:amplehighdegree} that, for high-enough degree, ampleness abuts essentially to the existence of certain low-degree elliptic and rational curves lying on the K3 surface. For some applications it might be ideal to discard the assumption on the degree, for this we included Propositions \ref{prop:l-delta} and \ref{prop:lminus2delta} where we put in evidence the technique to check ampleness without assumptions on the degree.

In Section \ref{section:gaussianmapsk3} we first consider the case of a \emph{general} K3 surface, and then prove Theorem C in \ref{thm:surjectivitygaussmapsK3}. The main application (Theorem D) is proved in \ref{thm:higherGauss}. We also consider the non-general case, where using the results of Section \ref{sec:amplenesscriterion} we give a criterion for surjectivity of gaussian maps in Theorem \ref{thm:highergaussmapsgeneral}. We conclude in Section \ref{section:furtherremarks} with further remarks and research directions.

\begin{thankyou}
It is a pleasure to thank my Ph.D. supervisor Kieran O'Grady for all his support during these years. Thanks also to Edoardo Sernesi and Gianluca Pacienza for their suggestions and comments on a first draft of this work. I would like to thank Fabrizio Anella and Simone Novario for all the useful conversations. Finally thanks also goes to the anonymous referees for their detailed comments and suggestions.
\end{thankyou}

\section{Gaussian maps and linear systems on the Hilbert scheme}\label{section:gaussianandhilbertscheme}

The following definitions of the higher Gaussian maps were first introduced by Wahl in \cite{Wahl90}. We will be mainly following \emph{loc.cit.} and also the survey \cite{Wahl92}. It is worth pointing out that there are some discrepancies in the literature about the definition of Gaussian maps.

\subsection{Higher Gaussian maps}

Let $(X,L)$ be a polarized projective variety. On $X\times X$ we consider the external product $L\boxtimes L:= \text{pr}_1^*L\otimes\text{pr}_2^*L$, where $\text{pr}_i$ is the projection to the $i$-th factor. Denote by $\mathcal{I}_{\Delta_X}$ the ideal sheaf of the diagonal $\Delta_X$ in $X\times X$. Recall that (cf. \cite[II.8]{Hart77}) there is a canonical isomorphism $(\mathcal{I}_{\Delta_X}/\mathcal{I}_{\Delta_X}^{2})|_
{\Delta_X} \cong \Omega^1_X$, hence for every $k\geq 0$ there is an isomorphism $(\mathcal{I}_{\Delta_X}^{k}/\mathcal{I}_{\Delta_X}^{k+1})|_{\Delta_X} \cong \mathrm{Sym}^k(\Omega^1_X)$. The global sections of $\mathcal{I}^k_{\Delta_X}(L\boxtimes L)$ define a filtration of $H^0(X,L\boxtimes L)$ due to the following exact sequence:
\begin{equation}\label{eq:gauss}
\begin{tikzcd}
0\ar[r] & \mathcal{I}_{\Delta_X}^{k+1}(L\boxtimes L)\ar[r] & \mathcal{I}_{\Delta_X}^{k}(L\boxtimes L)\ar[r] & \text{Sym}^k(\Omega^1_X)(2L)\ar[r] & 0.
\end{tikzcd}
\end{equation}
Where for the last identification we used that $(L\boxtimes L)|{\Delta_X} \cong \O_X(2L)$.
\begin{defn}
For any $k\geq 0$ the $k$-th Gaussian map $\gamma^k_L$ is the induced map on global sections
\[
\gamma^k_L: H^0(X\times X,  \mathcal{I}_{\Delta_X}^{k} (L\boxtimes L))\to H^0(X,\text{Sym}^k(\Omega^1_X)(2L))
\]
of the exact sequence \eqref{eq:gauss} above.
\end{defn}

\begin{rem}\label{rem:domaingauss}
By definition, for each $k\geq 1$, the domain of $\gamma^k_L$ is $\ker(\gamma^{k-1}_L)$. 
\end{rem}

The Gaussian maps are \emph{functorial}; if $f:X\to Y$ is a regular map between smooth varieties and $L$ is a line bundle on $Y$, then there exists a commutative diagram:
\begin{equation}
\begin{tikzcd}
H^0(Y\times Y,\mathcal{I}_{\Delta_Y}^k(L\boxtimes L))\ar[r,"\gamma^k_{L}"]\ar[d,"f^*"] & H^0(Y,\text{Sym}^k\Omega^1_Y(2L))\ar[d,"d^kf"]\\
 H^0(X\times X,\mathcal{I}_{\Delta_X}^k(f^*(L)\boxtimes f^*(L)))\ar[r,"\gamma_{f^*(L)}^k"] & H^0(X, \text{Sym}^k\Omega^1_X (2f^*(L)))
\end{tikzcd}
\end{equation}

where the vertical map on the right is induced by the pullback of (twisted) symmetric differentials.

Let us make an explicit description of the first and second Gaussian maps to convey a more geometric meaning to the definition. First notice that for $k = 0$, the morphism $\gamma^0_L$ is given by multiplication of sections, therefore by the decomposition
\begin{equation}
 H^0(X,L)\otimes H^0(X,L) \cong \bigwedge^2H^0(X,L)\oplus \text{Sym}^2(H^0(X,L))
\end{equation}

the alternating tensors are mapped to zero. The kernel of $\gamma^0_L$ is then given by $\ker(\gamma^0_L) \cong \bigwedge^2H^0(S,L)\oplus I_2(X,L)$ where we let
$I_2(X,L)$ be the kernel of $\gamma^0_L$ restricted to $\text{Sym}^2(H^0(S,L))$. If we identify $\text{Sym}^2(H^0(X,L))$ with $H^0(\P^N,\O_{\P^N}(2))$, then $\gamma^0_L$ is the restriction of sections. Note that surjectivity of $\gamma^0_L$ is equivalent to $X$ being generated by quadrics. Denote by $I_2(X,L)$ the space of quadrics containing $X$. We have that 
\[
\ker(\gamma^0_L) \cong \bigwedge^2H^0(X,L)\oplus I_2(X,L)
\]
is the domain of $\gamma^1_L$ (see Remark \ref{rem:domaingauss}), and the map $\gamma^1_L$ vanishes on symmetric tensors, therefore there is an isomorphism
\[
\ker(\gamma^1_L)\cong  \ker(\bigwedge^2H^0(X,L)\to H^0(X,\Omega^1_X(2L))) \oplus I_2(X,L).
\]
Hence the morphism $\gamma_L^1$ is surjective if and only if the map
\[
\bigwedge^2H^0(X,L)\to H^0(X,\Omega^1_X(2L))
\]
is surjective.

\begin{exmp}
Let $C$ be a curve and $L = K_C$, then the $k$-th Gaussian map is given as
\[
\gamma^k_{K_C}: H^0(C\times C,  \mathcal{I}_{\Delta^{k}_C}\otimes (K_C\boxtimes K_C))\to H^0(C, (k+2)K_C).
\]
When $k=0$, this map is given by multiplication on sections, when $k=1$ this map is the Wahl map, i.e. the morphism \eqref{eq:wahlmap}. We will usually write $\gamma^k_C$ when referring to $\gamma^k_{K_C}$.
\end{exmp}

\subsection{Linear Systems on the Hilbert Scheme}\label{sec:linearsystemsonthehilbertscheme}

Let $S$ be a regular surface (i.e. $H^1(S,\O_S) = 0$) and let $S^{[2]}$ be the Hilbert scheme of $2$ points on $S$, then by \cite[Theorem 6.2]{Fogarty73} the Picard group of $S^{[2]}$ is given as follows:
\begin{equation}\label{eq:pics2}
\text{Pic}(S^{[2]}) \cong \text{Pic}(S)\oplus \Z\delta,
\end{equation}
where $2\delta$ is the exceptional divisor of the blowup  $\mu:S^{[2]}\to S^{(2)}$ of the symmetric product along the diagonal $\Delta$. If $\pi: S\times S\to S^{(2)}$ is the quotient map, we will denote by $\Delta_S$ the diagonal in $S\times S$ and by $\Delta$ its image in $S^{(2)}$. For any line bundle $L$ on $S$, there exists a line bundle $L^{(2)}$ in $S^{(2)}$, such that $L\boxtimes L = \pi^*(L^{(2)})$. hence its pullback $\mu^*(L^{(2)})$ is a line bundle on $S^{[2]}$ (cf. \cite[Remark 2.1]{Scala20}).

\begin{notation}
For any line bundle $L$ on $S$ we will also denote by $L$ the induced line bundle on $S^{[2]}$ via the canonical identification \eqref{eq:pics2}.
\end{notation}

With the same notation as above, the universal family $\widetilde{S\times S}$ for $S^{[2]}$ is the blowup of the diagonal $\Delta_S$ in $S\times S$. We obtain the following commutative diagram:
\begin{equation}\label{eq:commutativediagram}
\begin{tikzcd}
\widetilde{S\times S}\ar[r, "q"]\ar[d, "p"] & S^{[2]}\ar[d,"\mu"]\\
S\times S\ar[r, "\pi"] & S^{(2)}
\end{tikzcd}    
\end{equation}
The morphism $q$ is the branched cover along the divisor $2\delta$ in $S^{[2]}$ and the vertical maps are blow-ups along the respective diagonals. We will  use the notation introduced throughout this section. In the following, we will also use the language of $\Z/2\Z$-equivariant sheaves, cf. Appendix \ref{section:equivariantsheaves} for the relevant definitions. 

\begin{notation}
The group $\Z/2\Z$ acts on $S\times S$ by permuting the factors. We will denote by $\varepsilon: \Z/2\Z\to \C^*$ the alternating character.
\end{notation}

\begin{prop}\label{prop:pushforwardpowersidealdiagonal}
With the same notation as above, we have:
\begin{enumerate}
    \item $\mu_*(-k\delta)\cong \mathcal{I}_{\Delta}^k\otimes\varepsilon^k$, and
    \item $\pi_*(\mathcal{I}_{\Delta_S}^k)\cong (\mathcal{I}_{\Delta}^k\otimes\varepsilon^k)\oplus (\mathcal{I}_{\Delta}^{k+1}\otimes\varepsilon^{k+1})$.
\end{enumerate}
\end{prop}
\begin{proof}
Item (1) is a particular case of a result of Scala, see \cite[Theorem 2.8]{Scala20}. For completeness we will give a proof. By Remark \ref{rem:sheafofinvariants} we have that 
\[
\mu_*(-k\delta)\cong \pi_*(\pi^*(\mu_*(-k\delta)))^{\Z/2\Z}.
\]
By flat base change applied to the diagram \eqref{eq:commutativediagram} we have that $\pi_*(\pi^*(\mu_*(-k\delta)))\cong \pi_*(q_*p^*(-k\delta))$.
Since $q$ is a branched cover we have $kE = q^*(k\delta)$. Let $E$ be the exceptional divisor on $\widetilde{S\times S}$, then by \cite[Lemma  4.3.16]{PAG1} there is a canonical isomorphism  $p_*(-kE) \cong \mathcal{I}_{\Delta_S}^k$. Therefore $p_*q^*(-k\delta)\cong \mathcal{I}_{\Delta_S}^k$. Taking the invariant part of the pushforward by $\pi$ yields the result.

To prove Item (2) we use that $q$ is a branched covering along the smooth divisor $2\delta$ in $S^{[2]}$, in particular \cite[Remark 4.1.7]{PAG1} we have that $q_*\O_{\widetilde{S\times S}}\cong \O_{S^{[2]}}\oplus \O_{S^{[2]}}(-\delta)$. Therefore:
\[
 \pi_*(\O_{S\times S}) \cong \pi_*p_*(\O_{\widetilde{S\times S}}) = \mu_*q_*(\O_{\widetilde{S\times S}}) \cong \mu_*(\O_{S^{[2]}}\oplus \O_{S^{[2]}}(-\delta)).
\]
The first isomorphism holds because $p$ is birational and $S\times S$ is normal. Also $\mu$ is birational and $S^{(2)}$ normal, hence $\mu_*(\O_{S^{[2]}}) \cong \O_{S^{(2)}}$. By Item (1) we get 
\begin{equation}\label{eq:pushforwardstructuresheaf}
\mu_*(\O_{S^{[2]}}\oplus \O_{S^{[2]}}(-\delta)) \cong \O_{S^{(2)}}\oplus (\mathcal{I}_{\Delta}\otimes\varepsilon).
\end{equation}
Since $q$ is a branched cover we have $kE = q^*(k\delta)$. Hence
\[
\pi_*(\mathcal{I}_{\Delta_S}^k) \cong \pi_*p_*(-kE) = \mu_*q_*(-kE) =  \mu_*q_*q^*(-k\delta),
\]
and we conclude using the projection formula and the isomorphism \eqref{eq:pushforwardstructuresheaf}.
\end{proof}

\begin{cor}\label{cor:pushpowersideal}
With the same notation as above, we have:
\[
\pi_*(\mathcal{I}_{\Delta_S}^k(L\boxtimes L))\cong ((\mathcal{I}_{\Delta}^k\otimes\varepsilon^k)\oplus (\mathcal{I}_{\Delta}^{k+1}\otimes\varepsilon^{k+1}))\otimes \O_{S^{(2)}}(L^{(2)}).
\]
Morever, this decomposition corresponds to the splitting in \eqref{eq:equivsheafsplits}.
\end{cor}
\begin{proof}
The result is a consequence of the projection formula and Proposition \ref{prop:pushforwardpowersidealdiagonal}, the decomposition must correspond to the splitting by uniqueness in Lemma \ref{lem:equivsheafsplits}.
\end{proof}

\begin{prop}\label{prop:symkomega1}
With the same notation, the higher direct images of $\pi$ vanish, and the restriction of $\pi$ to $\Delta_S$ induces an isomorphism: 
\[
\pi_*(\text{Sym}^k(\Omega^1_{\Delta_S}))\cong\text{Sym}^k(\Omega^1_\Delta)\cong \text{Sym}^k(\Omega^1_S).
\]
The action of $\Z/2\Z$ on the fibers of $\text{Sym}^k(\Omega^1_S)$ is given by the representation $\varepsilon^{k}\otimes \text{Sym}^k(\C^2)$.
\end{prop}
\begin{proof}
The map $\pi$ is finite, hence all higher direct images vanish. The quotient morphism is an isomorphism when restricted to the diagonal and therefore also for all the vector bundles supported on the diagonal. The last assertion is contained in the proof of Lemma 2.6. of \cite{Scala20}.
\end{proof}

\begin{cor}\label{cor:decompositioncohomology}
For every $p\geq 0$ and $k\geq 0$ there are natural identifications
\[
H^p(S\times S, \mathcal{I}_{\Delta_S}^k(L\boxtimes L)) \cong H^p(S^{[2]},L-k\delta)\oplus H^p(S^{[2]},L-(k+1)\delta).
\]
\end{cor}
\begin{proof}
By Proposition \ref{prop:symkomega1} the Leray spectral sequence degenerates and therefore 
\[
H^p(S\times S, \mathcal{I}_{\Delta_S}^k(L\boxtimes L))\cong H^p(S^{(2)}, \pi_*( \mathcal{I}_{\Delta_S}^k(L\boxtimes L)))
\]
for all $p\geq 0$. Corollary \ref{cor:pushpowersideal} and Item (1) in Proposition \ref{prop:pushforwardpowersidealdiagonal}  yield the result.
\end{proof}

Propositions \ref{prop:pushforwardpowersidealdiagonal} and \ref{prop:symkomega1} yield that the pushforward by  $\pi:S\times S\to S^{(2)}$ of the exact sequence \eqref{eq:gauss} when considering $L = \O_S$ is

\begin{equation}\label{eq:pushexactseq}
\begin{tikzcd}
0\ar[r] & (\mathcal{I}_{\Delta}^{k+1}\otimes\varepsilon^{k+1})\oplus (\mathcal{I}_{\Delta}^{k+2}\otimes\varepsilon^{k+2}) \ar[r] & \,\\
 & (\mathcal{I}_{\Delta}^{k}\otimes\varepsilon^{k})\oplus (\mathcal{I}_{\Delta}^{k+1}\otimes\varepsilon^{k+1})\ar[r] & \text{Sym}^k(\Omega^1_S)\ar[r] & 0.
\end{tikzcd}
\end{equation}

The maps are of $\Z/2\Z$-equivariant sheaves (for the trivial action on $S^{(2)}$). By Proposition \ref{prop:symkomega1} the action of $\Z/2\Z$ on the fibers of $\text{Sym}^k(\Omega^1_S)$ is given by $\varepsilon^{k}\otimes \text{Sym}^k(\C^2)$. Taking the $\Z/2\Z$-invariant part when $k$ is even, or the $\Z/2\Z$-anti-invariant part when $k$ is odd, gives the exact sequence

\begin{equation}\label{eq:invariantsinsym2}
\begin{tikzcd}
0\ar[r] & \mathcal{I}_\Delta^{k+2}\otimes\varepsilon^{k+2} \ar[r]  &\mathcal{I}_\Delta^{k}\otimes\varepsilon^k\ar[r] & \text{Sym}^k\Omega^1_S\ar[r] & 0.
\end{tikzcd}
\end{equation}

On the other hand, the divisor $2\delta$ is effective and represents the exceptional divisor of the blow-up $\mu: S^{[2]}\to S^{(2)}$, which is then identified with $\P(\Omega^1_S)$. The divisor exact sequence for $\P(\Omega^1_S)$ is then:
\begin{equation}\label{eq:divisorinhilb2}
\begin{tikzcd}
0\ar[r] & \O_{S^{[2]}}(-2\delta)\ar[r]  & \O_{S^{[2]}}\ar[r] & \O_{\P(\Omega^1_S)}\ar[r] & 0.
\end{tikzcd}
\end{equation}

The restriction of $\mu$ to $\P(\Omega^1_S)$ is a $\P^1$-bundle over $S$. Moreover $-\delta|_{\P(\Omega^1_S)} = \xi$ is the tautological class on $\P(\Omega^1_S)$ and for every $k\geq 0$ there is an isomorphism
$\mu_*(\O_{\P(\Omega^1_S)}(k\xi))\cong \text{Sym}^k\Omega^1_S$
(cf. \cite[Appendix A]{PAG1}).
By Proposition \ref{prop:pushforwardpowersidealdiagonal}, if we tensor the exact sequence \eqref{eq:divisorinhilb2} with $\O_{S^{[2]}}(-k\delta)$ and push it forward by $\mu$, we obtain:

\begin{equation}\label{eq:divisorinsym2}
\begin{tikzcd}
0\ar[r] & \mathcal{I}_\Delta^{k+2}\otimes\varepsilon^{k+2} \ar[r]  &\mathcal{I}_\Delta^{k}\otimes\varepsilon^k\ar[r] & \text{Sym}^k\Omega^1_S\ar[r] & 0.
\end{tikzcd}
\end{equation}

The diagram \eqref{eq:commutativediagram} is commutative, hence we get the following.

\begin{prop}\label{prop:identification}
The induced exact sequence \eqref{eq:invariantsinsym2} given by taking invariants and the induced exact sequence \eqref{eq:divisorinsym2} given by the pushforward of $\mu$ are the same.
\end{prop}

All the previous computations give us a sufficient condition for the surjectivity of higher Gaussian maps using the cohomology of line bundles on the Hilbert scheme.

\begin{thm}\label{thm:vanish}
Let $L$ be a line bundle on a regular surface $S$. If the cohomology group $H^1(S^{[2]}, L-(k+2)\delta)$ vanishes, then the $k$-th Gaussian map $\gamma^k_L$ is surjective. 
\end{thm}
\begin{proof}
The $k$-th Gaussian map for the line bundle $L$ is by definition the morphism on global sections
\begin{equation}\label{eq:gaussiansurface}
\gamma^k_L: H^0(S\times S,  \mathcal{I}_{\Delta_S}^{k} (L\boxtimes L))\to H^0(S,\text{Sym}^k(\Omega^1_S)(2L))
\end{equation}
induced by the exact sequence \eqref{eq:gauss}. By definition the global sections of $\mathcal{I}_{\Delta_S}^{k} (L\boxtimes L))$ and $\pi_*(\mathcal{I}_{\Delta_S}^{k} (L\boxtimes L)))$ are the same. Hence, by pushing forward \eqref{eq:gauss} via $\pi$, yields

\begin{equation}
\begin{tikzcd}
0\ar[r] & (\mathcal{I}_{\Delta}^{k+1}\otimes\varepsilon^{k+1})(L^{(2)})\oplus (\mathcal{I}_{\Delta}^{k+2}\otimes\varepsilon^{k+2})(L^{(2)}) \ar[r] & \,\\
 & (\mathcal{I}_{\Delta}^{k}\otimes\varepsilon^{k})(L^{(2)})\oplus (\mathcal{I}_{\Delta}^{k+1}\otimes\varepsilon^{k+1})(L^{(2)})\ar[r] & \text{Sym}^k(\Omega^1_S(2L))\ar[r] & 0.
\end{tikzcd}
\end{equation}
By Proposition \ref{prop:symkomega1}, the action of $\Z/2\Z$ on the fibers of $\mathrm{Sym}^k(\Omega^1_S)$ is via the representation $\varepsilon^k\otimes\mathrm{Sym}^k(\C^2)$. Therefore necessarily the linearization on $\mathrm{Sym}^k(\Omega^1_S)$ is different from the one of $(\mathcal{I}_{\Delta}^{k+1}\otimes\varepsilon^{k+1})$, hence by taking sections of the exact sequence above gives that  $H^0(S^{[2]},L-(k+1)\delta)\subseteq \ker(\gamma^k_L)$ for each $k\geq 0$. This implies that the morphism \eqref{eq:gaussiansurface} is surjective if and only if the restriction
\begin{equation}
\gamma^k_L: H^0(S^{(2)}, \pi_*(\mathcal{I}_{\Delta}^{k}\otimes\varepsilon^{k})(L^{(2)}))\to H^0(S,\text{Sym}^k(\Omega^1_S)(2L))
\end{equation}
is surjective. The identification given in Proposition \ref{prop:identification} gives that the restriction above is the same as the map in cohomology
\[
\begin{tikzcd}
H^0(S^{(2)},\mathcal{I}_\Delta^{k}\otimes\varepsilon^{k}(L^{(2)}))\ar[r,"\gamma^k_L"]& H^0(\Omega^1_S,\text{Sym}^k\Omega^1_S(2L))\ar[r] & H^1(S^{(2)},\mathcal{I}_\Delta^{k+2}\otimes\varepsilon^{k+2}(L^{(2)})).
\end{tikzcd}    
\]

Since $S^{(2)}$ has rational singularities we have that $\mu_*\O_{S^{[2]}} = \O_{S^{(2)}}$. By the projection formula we get
\[
H^1(S^{(2)},\mathcal{I}_\Delta^{k+2}\otimes\varepsilon^{k+2}(L^{(2)}))\cong H^1(S^{[2]}, L-(k+2)\delta).
\]
Hence if the last group vanishes the morphism $\gamma^k_L$ is surjective as we wanted to prove.
\end{proof}

A good knowledge of the nef cone of the surface $S$ will let us show surjectivity of a higher Gaussian map by applying standard vanishing theorems for big and nef line bundles. Let us mention an immediate corollary of Theorem \ref{thm:vanish} concerning Gaussian maps for the hyperplane line bundle on $\P^2$.

\begin{cor}
Let $H$ be the hyperplane line bundle on $\P^2$. If $a\geq k-1$, then $\gamma^k_{aH}$ is surjective.
\end{cor}
\begin{proof}
In \cite[Theorem 4.1]{LQZ03} it is shown that the nef cone of $\P^{2[2]}$ is generated by the classes $H$ and $H-\delta$, hence the class $(3+a)H-(k+2)\delta$ is big and nef if and only if $a\geq k+1$. An application of Kawamata-Viehweg vanishing and Theorem \ref{thm:vanish} yields the result.
\end{proof}

\section{An ampleness criterion}\label{sec:amplenesscriterion}

Let $(S,L)$ be an arbitrary polarized K3 surface. Recall that $L$ is called $k$-very ample, for an integer $k\geq 0$, if for any $0$-dimensional subscheme $Z$ of length $k+1$ the restriction map $H^0(S,L)\to H^0(S,L\otimes\O_Z)$ is surjective. There is a generalization of a celebrated result of Saint-Donat \cite{SD1974} due to Knutsen that characterizes $k$-very ampleness \emph{numerically}.

\begin{thm}\cite[Theorem 1.1]{Knutsen01}\label{thm:knutsen}
Let $L$ be a big and nef line bundle on a K3 surface and $k\geq 0$ an integer. The following conditions are equivalent:
\begin{enumerate}
    \item $L$ is $k$-very ample.
    \item $L^2\geq 4k$ and there exists no effective divisor $D$ satisfying the conditions:
    \begin{equation}\label{eq:knutsen}
    2D^2\leq L\cdot D\leq D^2 + k + 1 \leq 2k + 2.    
    \end{equation}
\end{enumerate}
\end{thm}

Let $S$ be a K3 surface. Via the identification \eqref{eq:pics2}, all divisors in $\Pic(S^{[2]})$ are of the form $D+a\delta$, where $D\in \Pic(S)$ and $a\in \Z$. We will now show an immediate corollary of Theorem \ref{thm:knutsen}.

\begin{cor}\label{cor:veryample}
Let $L$ be an ample divisor on $S$, then for $m \gg 0$ the divisor $mL-\delta$ is very ample on $S^{[2]}$.
\end{cor}
\begin{proof}
The main theorem of \cite{CatGoe90} says that the divisor $mL-\delta$ is very ample if and only if $mL$ is $2$-very ample. Using the characterization of $2$-very ampleness in Theorem \ref{thm:knutsen}, this occurs if and only if $(mL)^2\geq 8$ and there is no effective divisor $D$ such that
\begin{equation}\label{eq:2veryamplenness}
2D^2\leq mL\cdot D\leq D^2 + 3\leq 6.    
\end{equation}
Since $D$ is effective and $L$ ample, we always have $L\cdot D \geq 1$. On the other hand, equation \eqref{eq:2veryamplenness} implies that $mL\cdot D\leq 6$. So, whenever $m\geq 7$ the equation \eqref{eq:2veryamplenness} cannot hold and surely $(mL)^2\geq 8$. The thesis follows.
\end{proof}
\begin{rem}\label{rem:nefconeisfulldimensional}
As a consequence of Corollary \ref{cor:veryample}, we have that
$\Nef(S^{[2]})\cap \mathrm{span}_\R(L,\delta)$ is a two dimensional cone. Indeed, it is a cone since $\Nef(S^{[2]})$ is a cone and $\text{span}_\R(L,\delta)$ is a linear subspace; the class $mL-\delta$ for $m\gg 0$ is in $\Nef(S^{[2]})$ by Corollary \ref{cor:veryample}. Finally notice that $L$ is big and nef, because is the pullback of the ample divisor $L^{(2)}$ under the blowup map, hence $L$ is on the boundary of $\Nef(S^{[2]})$.
\end{rem}

We will use the explicit characterization of the Nef cone for moduli spaces of stable objects on K3 surfaces given in \cite{BMMMP14}. For this we need to introduce some notation. For any K3 surface $S$, the Hilbert scheme $S^{[2]}$ is an example of a \emph{hyperk\"ahler variety} \cite{Beau83}. For every hyperk\"ahler manifold $X$ the second integral cohomology $H^2(X,\Z)$ is endowed with a quadratic form, called the Beauville-Bogomolov form, which is denoted by $q_X$ (see \cite{Beau83} for details). For $S^{[2]}$, there is an isomorphism of lattices preserving the Hodge structure
\[
H^2(S^{[2]},\Z) \cong H^2(S,\Z) \oplus \Z\cdot\delta 
\]
where $q_{S^{[2]}}(\delta) = -2$ and the lattice structure on $H^2(S,\Z)$ is given by the intersection pairing. The divisibility of an element $\kappa$ in $H^2(S^{[2]},\Z)$, denoted by $\mathrm{div}(\kappa)$, is defined as the positive generator of the ideal $q(\kappa,H^2(S^{[2]},\Z))$ in $\Z$.

\begin{defn}
Let $S$ be a projective K3 surface. A class $\kappa\in\Pic(S^{[2]})$ is called a $(-2)$-class if $q(\kappa) = -2$; it is called a $(-10)$-class if $q(\kappa) = -10$ and also $\mathrm{div}(\kappa) = 2$.
\end{defn}

The \emph{positive cone} $\mathrm{Pos}(X)$ of a hyperk\"ahler manifold $X$ is by definition the connected component of $\{\alpha \in H^{1,1}(X,\R) : q(\alpha) >0\}$ containing an ample class. The stable base locus of a line bundle $L$ on $X$ is the intersection of the base loci $|kL|$ for all positive integers $k$. We say that $L$ is movable if its stable base locus has codimension at least $2$ on $X$. We denote by $\mathrm{Mov}(X)$ the cone spanned by movable classes.

\begin{thm}[\cite{BMMMP14}]\label{thm:bayermacri}
Let $S$ be a projective $K3$ surface. Then:
\begin{enumerate}
    \item The interior of the movable cone of $S^{[2]}$ is the connected component of
    \[
    \mathrm{Pos}(S^{[2]}) \setminus \bigcup_{\kappa\in \Pic(S^{[2]})\,:\, q(\kappa) = -2} \kappa^\bot
    \]
    that contains the class of an ample divisor.
    \item The ample cone is the connected component of
    \[
    \mathrm{Mov}(S^{[2]}) \setminus \bigcup_{\kappa\in \Pic(S^{[2]}) \,:\, q(\kappa) = -10,\, \mathrm{div}(\kappa) = 2} \kappa^\bot
    \]
    that contains the class of an ample divisor.
\end{enumerate}
\end{thm}

\begin{cor}\label{cor:kappaclasses}
Let $(S,L)$ be a polarized $K3$ surface and $a>0$ an integer. Assume $q(L-a\delta)\geq 0$, then the class $L-a\delta$ in $S^{[2]}$ is:
\begin{enumerate}
    \item In the interior of the movable cone if and only if there exists no class $\kappa = D -b\delta$ satisfying:
    \begin{enumerate}
        \item  $q(\kappa) =-2$ and $b>0$.
        \item $D\in \Pic(S)$ is a non-zero effective divisor.
        \item $q(L-a\delta,\kappa)\leq 0$. 
    \end{enumerate}
    \item Ample if and only if it is movable and there exists no class $\kappa = 2D -b\delta$ satisfying:
    \begin{enumerate}
        \item  $q(\kappa) =-10$ and $b>0$.
        \item $D\in \Pic(S)$ is a non-zero effective divisor.
        \item $q(L-a\delta,\kappa)\leq 0$.
    \end{enumerate}
\end{enumerate}
\end{cor}
\begin{rem}
If $a<0$, the class $L-a\delta$ cannot be ample, since it will have negative intersection with any curve in the exceptional divisor.
\end{rem}
\begin{proof}[Proof of Corollary \ref{cor:kappaclasses}]
We just need to prove the implication from right to left, i.e. that $L-a\delta$ is in the interior of the movable cone, for this we will use the characterization in Theorem \ref{thm:bayermacri}. Let $\kappa = D - b\delta$ be a $(-2)$-class in $\Pic(S^{[2]})$, with $b\in\Z$. We can assume $b\geq 0$ since $\kappa^\perp = (-\kappa)^\perp$. If we prove that $D$ is effective then by hypothesis $q(L-a\delta,\kappa)>0$ and therefore $L-a\delta$ will be in the interior of the movable cone. Since $D^2 = 2b^2-2\geq -2$, by Riemann-Roch either $D$ or $-D$ is effective. On the other hand $L$, as divisor in $S^{[2]}$, is big and nef (Remark \ref{rem:nefconeisfulldimensional}) and therefore by Theorem \ref{thm:bayermacri} we must have $0\leq q(L,\kappa) = L\cdot D$. Hence $D$ must be effective and (1) follows. 

For the proof of (2), let $\kappa$ be a $(-10)$-class in $\Pic(S^{[2]})$. Since the lattice $H^2(S,\Z)$ is unimodular, $\kappa$ has divisibility $2$ if and only if it is of the form $\kappa = 2D-b\delta$; we can further assume $b\geq 0$ as before. Moreover, if $b=0$, then $\kappa = 2D$ and therefore $q(\kappa) = 4D^2$ cannot be $-10$, hence we have $b>0$. Using Theorem \ref{thm:bayermacri} again, we are left to check that  $D$ must be effective. Using $q(\kappa) = -10$ we get that
\[
D^2 = \frac{b^2-1}{2}-2\geq -2
\]
with $b\geq 1$, the same arguments as in the first case yield the result.
\end{proof}

In general there exist an infinite number of $(-2)$-classes and $(-10)$-classes on a Hyperk\"ahler manifold of  $K3^{[2]}$-type. Corollary \ref{cor:kappaclasses} will allow to show that with respect to a fixed divisor in $S^{[2]}$ they can be "bounded", the precise meaning is explained in the following.

\begin{thm}\label{thm:alglineal-2}
Let $\kappa = D-b\delta$ be a $(-2)$ class as in Corollary \ref{cor:kappaclasses}. If $q(L-a\delta)>0$ and $q(L-a\delta, \kappa)\leq0$, then 
\begin{equation}\label{eq:-2classes}
    0 < b^2\leq \frac{d}{d-a^2} \,\,\,\,\,\,\text{and}\,\,\,\,\,\, 0 < L\cdot D\leq 2ab,  
\end{equation}
where $L^2 = 2d$.
\end{thm}
\begin{proof}
Define the integers $x:= D^2$ and $y:= L\cdot D$. In our notation we have $q(L-a\delta)>0$ if and only if $d-a^2>0$. The sublattice spanned by $L$ and $D$ in $\Pic(S)$ \emph{cannot} be positive definite by the Hodge Index Theorem. Hence
\begin{equation}\label{eq:hodgeindex}
    2dx-y^2\leq 0.
\end{equation}
Since $q(\kappa) = -2$, then $x = 2b^2-2$. The fact that $\kappa$ intersects non-positively $L-a\delta$ yields the inequality
\begin{equation}\label{eq:intersecneg}
    y-2ab\leq 0.
\end{equation}
Notice that $y>0$, because $L$ is ample by hypothesis and $D$ is effective by Corollary \ref{cor:kappaclasses}, this gives the second inequality in \eqref{eq:-2classes}. Squaring \eqref{eq:intersecneg} gives $y^2\leq 4a^2b^2$ and substituting $x$ in \eqref{eq:hodgeindex} implies the inequality
\[
4d(b^2-1)\leq y^2\leq 4a^2b^2.
\]
An algebraic manipulation yields the result.
\end{proof}

There is an analogous result for $(-10)$-classes, we omit the proof since it is very similar to that of Theorem \ref{thm:alglineal-2}.

\begin{thm}\label{thm:alglineal-10}
Let $\kappa = 2D-b\delta$ be a $(-10)$ class as in Corollary \ref{cor:kappaclasses}. If $q(L-a\delta)>0$ and $q(L-a\delta, \kappa)\leq0$, then 
\begin{equation}\label{eq:-10classes}
    0 < b^2\leq \frac{5d}{d-a^2} \,\,\,\,\,\,\text{and}\,\,\,\,\,\, 0 < L\cdot D\leq ab \,\,\,\,\,\,\text{with}\,\,\,\,\,\, \frac{b^2-5}{2}\equiv 0\mod 2
\end{equation}
where $L^2 = 2d$.
\end{thm}

\begin{rem}\label{rem:boundD2}
With the notation of Theorems \ref{thm:alglineal-2} and \ref{thm:alglineal-10},  whenever $x:=D^2>0$ we can use Equation \eqref{eq:hodgeindex} to improve the second inequality in Equation \eqref{eq:-2classes} and \eqref{eq:-10classes} to be $\sqrt{2dx}\leq L\cdot D$. Notice that $x = 2b^2-2$ in the first Equation and $x = \frac{b^2-5}{2}$ in the second one.
\end{rem}

Whenever the degree of the polarization is high enough with respect to $a^2$, that the only possible cases to check in equations \eqref{eq:-2classes} and \eqref{eq:-10classes} correspond to effective divisors in $S$ with self-intersections $-2$ or $0$, therefore we obtain a cleaner result.

\begin{cor}\label{cor:movablehighdegree}
Let $(S,L)$ be a polarized K3 surface with $L^2 = 2d$. Suppose $d>\frac{4a^2}{3}$, then the class $L-a\delta$ is movable if and only if there are no effective divisors $D$ in $S$ such that $D^2 = 0$ and $L\cdot D\leq 2a$.
\end{cor}
\begin{proof}
Use the same notation as in Remark \ref{rem:boundD2}. If $x :=D^2\neq 0$, then we have the bound
\[
\sqrt{2dx} = 2\sqrt{d(b^2-1)} \leq L\cdot D \leq 2ab. 
\]
Notice that $b>1$ by hypothesis on $x$. We obtain the following inequality
\[
\frac{4a^2}{3} < d\leq \frac{a^2b^2}{(b^2-1)}
\]
and this is a contradiction since $b>1$. When $x=0$ we have $b=1$ and the possible cases are the ones stated in the Lemma. Since $L$ is ample, the case $x=-2$ cannot occur.
\end{proof}

The proof of the following corollary is the same as Corollary \ref{cor:movablehighdegree} and we will omit it.

\begin{cor}\label{cor:amplehighdegree}
Let $(S,L)$ be a polarized K3 surface with $L^2 = 2d$. Suppose $d>\frac{9a^2}{4}$, then the class $L-a\delta$ is big and nef (resp. ample) if and only if it is movable and there are no effective divisors $D$ in $S$ such that $D^2 = -2$ and $L\cdot D < a$ (resp. $L\cdot D \leq a$).
\end{cor}

\subsection{Examples}

For some applications the hypothesis on the degree given in Corollaries \ref{cor:movablehighdegree} and \ref{cor:amplehighdegree} are strong. However, it is possible to get rid of them, with the cost of getting more cases. In the following, we will give some examples for classes of the form $L-a\delta$ with $a\leq 2$ to give an idea of how to obtain these types of results.

\begin{prop}\label{prop:l-delta}
Let $(S,L)$ be a polarized K3 surface and suppose $L^2\geq 4$. The divisor $L-\delta$ is:
\begin{enumerate}
    \item In the interior of the movable cone if and only if there are no effective divisors $D$ in $S$ such that $D^2 = 0$ and $L\cdot D\leq 2$. 
    \item Ample if and only if it is movable and there are no effective divisors $D$ in $S$ such that $D^2=-2$ and $L\cdot D = 1$.
\end{enumerate}
\end{prop}
\begin{rem}
Notice that if $L^2 = 2$, then $q(L-\delta) = 0$ and hence $L-\delta$ cannot be in the interior of the movable cone. 
\end{rem}
\begin{proof}[Proof of Proposition \ref{prop:l-delta}]
By Corollary \ref{cor:movablehighdegree} we have that $L-\delta$ is in the interior of the movable cone if and only if there are no effective divisors $D$ with $D^2 = 0$ and $L\cdot D\leq 2$. By Corollary \ref{cor:amplehighdegree}, we have proven the proposition unless $L^2=4$.

Let us prove the case $L^2 = 4$. Using the notation in Theorem \ref{thm:alglineal-10}, with $d = 2$, suppose that $\kappa =2D-b\delta$ is a $(-10)$-class such that $q(L-\delta,\kappa)\leq 0$, then we must have
\begin{equation}\label{eq:-10classesa=1}
    0 < b^2\leq \frac{5d}{d-1} \,\,\,\,\,\,\text{and}\,\,\,\,\,\, 0 < L\cdot D \leq b \,\,\,\,\,\,\text{with}\,\,\,\,\,\, \frac{b^2-5}{2}\equiv 0\mod 2
\end{equation}
The trivial case, i.e. $b=1$, implies that $L-\delta$ is not on the interior of the ample cone if there exists an effective divisor $D$ such that $D^2 = -2$ and $L\cdot D = 1$.
The only non-trivial solution\footnote{Notice that $b=2$ is always a solution but it does not satisfy the last hypothesis in \eqref{eq:-10classesa=1}.} is given by $b=3$. In such case the bound on $L\cdot D$ given in the Remark \ref{rem:boundD2} implies that we must have $L\cdot D=3$. If we let $D$ denote a class with those invariants, then $(L-D)^2 = 0$ and $L\cdot(L-D) = 1$, this cannot hold since we assume $L-\delta$ to be movable.
\end{proof}

\begin{prop}\label{prop:lminus2delta}
Let $(S,L)$ be a polarized K3 surface and suppose $L^2\geq 10$. The divisor $L-2\delta$ is: 
\begin{enumerate}
    \item Movable if and only if there are no effective divisors $D$ such that $D^2 = 0$ and $L\cdot D \leq 4$.
    \item Ample if and only if it is movable, we have $L^2\neq 10$ and there does not exist an effective class $D$ such that $D^2 = -2$ and $L\cdot D \leq 2$. Furthermure, if $d=6$, there should be no effective divisor $D$ such that $D^2 = 2$ and $L\cdot D = 6$. If $d=9$ we require that $L$ is not of the form $3D$ for $D^2 = 2$.
\end{enumerate}
\end{prop}
\begin{proof}

We will use the same notation as in Corollaries  \ref{cor:movablehighdegree} and \ref{cor:amplehighdegree}, in particular we let $L^2 = 2d$. By Corollary \ref{cor:movablehighdegree} the only case left to check is $d = 5$. Using the notation in Theorem \ref{thm:alglineal-2}, suppose that $\kappa = D-b\delta$ is a $(-2)$-class such that $q(L-2\delta,\kappa)\leq 0$, then we must have
\begin{equation}\label{eq:-2classesa=2}
    0 < b^2\leq 5 \,\,\,\,\,\,\text{and}\,\,\,\,\,\, 0 < L\cdot D\leq 4b,  
\end{equation}

When the solution is $b=1$ we get that $L-2\delta$ it is not on the interior of the movable cone if there are effective divisors $D$ such that $D^2 = 0$ and $L\cdot D\leq 4$. The only non trivial solution is $b = 2$, then with the notation of Theorem \ref{thm:alglineal-2} we must have $x =6$. Since $\sqrt{60} < L\cdot D\leq 8$ by the Remark \ref{rem:boundD2}, the only possible case occurs when $L\cdot D=8$. In this case we notice that the class $L-D$ satisfies $(L-D)^2 = 0$ and $L\cdot(L-D) = 2$. This implies that $L-2\delta$ is not in the movable cone by the case $b=1$ in Equation \ref{eq:-2classesa=2}, and this proves the first assertion in the proposition.

For $(-10)$-classes, Corollary \ref{cor:amplehighdegree} proves the proposition unless $d\leq 9$. Suppose that $\kappa = 2D-b\delta$ is a $(-2)$-class such that $q(L-2\delta,\kappa)\leq 0$, then we must have 
\begin{equation}\label{eq:-10classesb=2}
    0 < b^2\leq \frac{5d}{d-4} \,\,\,\,\,\,\text{and}\,\,\,\,\,\, 0 < L\cdot D\leq 2b \,\,\,\,\,\,\text{with}\,\,\,\,\,\, \frac{b^2-5}{2}\equiv 0\mod 2
\end{equation}
 
In case $b=1$ we have $D^2 = -2$ with $L\cdot D\leq 2$. Hence $L-2\delta$ is not in the interior of the ample cone if there exists an effective class $D$ such that $D^2 = -2$ and $L\cdot D \leq 2$. From now on we will assume there are no such classes. We need to analyse case-by-case for the possible non-trivial solutions of $b$: 
\begin{enumerate}
    \item $d=5$: The class $2L-5\delta$ is a $(-10)$-class and $q(L-2\delta,2L-5\delta) = 0$, hence in this case the divisor $L-2\delta$ cannot be ample.
    \item $d=6$: The only possible value is $b=3$, then $D^2 = 2$ and as before we must have $\sqrt{24}\leq L\cdot D \leq 6$. If $L\cdot D=5$ we get that $(L-2D)^2=0$ and $L\cdot(L-2D) = 2$ but this is excluded by the previous part on $(-2)$-classes. In the last case we need to add a further restriction: there are no effective classes $D$ such that $D^2 = 2$ and $L\cdot D = 6$.
    \item $d=7$: the only possible value is $b=3$, then $\sqrt{28}\leq L\cdot D\leq 6$ implies that $L\cdot D=6$. Here the class $L-2D$ is a $(-2)$ curve such that $L\cdot(L-2D) = 2$, this is excluded by our assumption.
    \item $d=8$: the only possible value is $b=3$, then $L\cdot D=6$ and $3D-L$ is a $(-2)$ curve such that $L\cdot (3D-L) = 2$, this is excluded by our assumption.
    \item $d=9$: the only possible value is $b=3$ and $L\cdot D=6$. Here $L\cdot(L-3D) = 0$, with $D^2=2$. Hence $L$ is not ample \emph{unless} $L = 3D$.
\end{enumerate}
\end{proof}

\begin{rem}
Notice that the case $L^2 = 12$ adds a genuine new restriction because if there exists a divisor $D$ such that $D^2 = 2$ and $L\cdot D = 6$, then the lattice generated by $L$ and $D$ does not represent $-2$ and $0$ non-trivially.
\end{rem}

\section{Gaussian maps for K3 surfaces and an application}\label{section:gaussianmapsk3}

\subsection{K3 surfaces of Picard rank 1}
Let $(S,L)$ be a polarized K3 surface such that $\Pic(S)$ is generated by $L$. In this case, as a consequence of the very influential paper \cite{BMMMP14}, there is a very explicit description of the nef and movable cones for $S^{[2]}$. The nef cone of $S^{[2]}$ is generated by two rays: the first one is given by the induced divisor $L$ and the other one is determined by a Pell equation.

\begin{thm}\cite[Theorem 13.1]{BMMMP14}\label{thm:bayermacrigeneralk3}
Let $(S,L)$ be a polarized K3 surface such that $\Pic(S) = \Z L$ and let $L^2 = 2d$.
\begin{enumerate}
    \item Assume the equation $x^2 -4dy^2 = 5$ has no integral solutions.
    \begin{enumerate}
        \item If $d$ is a perfect square then $\Nef(S^{[2]}) = \langle L, L-\sqrt{d}\delta \rangle$.
        \item Else, the equation $x^2-dy^2 = 1$ has a minimal integral solution\footnote{meaning $a$ is minimal and $a,b>0$.} $(a,b)$, and $\Nef(S^{[2]}) = \langle L, L-d\frac{b}{a}\delta \rangle$.
    \end{enumerate}
    \item If the equation $x^2 -4dy^2 = 5$ has a minimal integral solution $(a,b)$ then $\Nef(S^{[2]}) = \langle L, L-2d\frac{b}{a}\delta\rangle$.
\end{enumerate}
\end{thm}

We use the theorem above, together with some vanishing results for the cohomology of Hyperk\"ahler manifolds, to infer the surjectivity of higher Gaussian maps in this case. The following lemma is well-known although we couldn't find a reference for it.

\begin{lem}\label{lem:cohomologyisotropicclass}
Let $D$ be a primitive\footnote{meaning that is not a multiple of another line bundle} line bundle on $S^{[2]}$ such that $D$ is nef and $q_{S^{[2]}}(D) = 0$, then the complete linear system associated with $D$ induces a fibration $\pi:S^{[2]}\to\P^2$. Moreover $h^i(S^{[2]},D) = 0$ for all $i>0$.
\end{lem}
\begin{proof}
The first part is Theorem 1.5 in \cite{BMMMP14}. To prove the second part let $\pi:S^{[2]}\to\P^2$ be the induced map given by $|D|$. Since $D$ is primitive we have $D=\pi^*(\O_{\P^2}(1))$. By \cite{Mat05} we have $R^i\pi_*\O_{S^{[2]}}\cong \Omega^i_{\P^2}$, therefore by projection formula
$R^p\pi_*\O_{S^{[2]}}(D) \cong \Omega^p_{\P^2}(1)$.
Bott vanishing for $\P^2$ yields $H^i(\Omega^p_{\P^2}(1)) = 0$ for $i>0$ and $H^0(\Omega^p_{\P^2}(1)) = 0$ by the Euler exact sequence. Hence the Leray spectral sequence degenerates and therefore $h^i(S^{[2]},D) = 0$ for all $i>0$.
\end{proof}

Recall also Verbitsky's vanishing result for the cohomology of line bundles on Hyperk\"ahler varieties. 

\begin{thm}\cite[Theorem 1.6]{Verbitsky07}\label{thm:vanishingverbitsky}
Let $X$ be a Hyperk\"ahler variety of dimension $2n$ and $L$ a line bundle on $X$. Denote its K\"ahler cone by $\KC_X\subseteq H^{1,1}(X,\R)$ and by $-\KC_X^\vee$ its opposite cone. Then the following hold:
\begin{enumerate}
    \item If $L\in \KC_X$, then $H^i(X,L) = 0$ for $i>n$.
    \item If $L\in -\KC_X^\vee$, then $H^i(X,L) = 0$ for $i<n$.
    \item If $L$ does not lie in $\KC_X\cup -\KC_X^\vee$, then $H^i(X,L) = 0$ for $i\neq n$
\end{enumerate}
\end{thm}

\begin{thm}\label{thm:surjectivitygaussmapsK3}
Let $(S,L)$ be a polarized K3 surface of degree $2d$ with $\Pic(S) = \Z L$, then $\gamma^k_L$ is surjective for all $k>0$.
\end{thm}
\begin{proof}
By Theorem \ref{thm:vanish}, we need to show that the groups $H^1(S^{[2]}, L-(k+2)\delta)$ vanish for every $k>0$. Assume first that the equation $x^2-4dy^2=5$ has no integral solutions, then there are two possible cases:
\begin{itemize}
    \item If $d = t^2$ is a perfect square we are in the case (1).(a) of Theorem \ref{thm:bayermacrigeneralk3}, hence $\Nef(S^{[2]}) = \langle L, L-t\delta \rangle$ and by direct computation we obtain $\Nef(S^{[2]})^\vee = \langle \delta, L-t\delta\rangle$. Assume $k+2>t$, then the class $L-(k+2)\delta$ does not lie in $\KC_{S^{[2]}}\cup -\KC_{S^{[2]}}^\vee$, therefore by Theorem \ref{thm:vanishingverbitsky} we have $H^1(S^{[2]}, L-(k+2)\delta) = 0$. If $k+2 < t$, then the class $L-(k+2)\delta$ is ample because the ample cone is the interior of $\langle L, L-t\delta\rangle$. By Kawamata-Viehweg vanishing theorem we obtain  $H^1(S^{[2]}, L-(k+2)\delta) = 0$ . Finally, if $k+2 = t$, then $H^1(S^{[2]}, L-(k+2)\delta) = 0$ by Lemma \ref{lem:cohomologyisotropicclass} above.
    
    \item If $d$ is not a perfect square, we are in case (1).(b) of Theorem \ref{thm:bayermacrigeneralk3}. Hence the equation $x^2-dy^2 = 1$ has a minimal solution $(a,b)$ and $\Nef(S^{[2]}) = \langle L, L-d\frac{b}{a}\delta \rangle$. One computes $\Nef(S^{[2]})^\vee = \langle \delta, L-\frac{a}{b}\delta\rangle$. If  $k+2 \leq \frac{db}{a}$, then the class $L-(k+2)\delta$ is ample and $H^1(S^{[2]}, L-(k+2)\delta) = 0$ by Kawamata-Viehweg vanishing theorem. If $k+2> \frac{a}{b}$ then $H^1(S^{[2]}, L-(k+2)\delta) = 0$ by the third case in Theorem \ref{thm:vanishingverbitsky}. The only possible case left is
    \[
    \frac{db}{a} < k+2 \leq \frac{a}{b}.
    \]
    Assume first that $k+2 = \frac{a}{b}$, this gives 
    \[
    q(L - \frac{a}{b}\delta) = 2d - 2\frac{a^2}{b^2} = -\frac{2}{b^2},
    \]
    which is an integer. Since $(a,b)$ is a solution for the equation $x^2-dy^2 = 1$ we must have that $b=1$ and therefore  $a^2=1 + d^2$. The only solution to the last equation is $a=1$ and $d =0$, this cannot happen since by assumption $L$ is ample. Assume now that $k+2 <\frac{a}{b}$, notice that $\frac{a}{b}-\frac{db}{a} =\frac{1}{ab}<1$, therefore there are no integers between these bounds and this case cannot occur.
\end{itemize}

Suppose $x^2-4dy^2=5$ has integer solutions and let $(a,b)$ be a minimal solution, then by the case (2) of Theorem \ref{thm:bayermacrigeneralk3} we have $\Nef(S^{[2]}) = \langle L, L-2d\frac{b}{a}\delta \rangle$ and $\Nef(S^{[2]})^\vee = \langle \delta, L-\frac{a}{2b}\delta\rangle$. The only possible case where the vanishing of $H^1(S^{[2]}, L-(k+2)\delta)$ is not guaranteed is when
\[
\frac{2db}{a} < k+2 \leq \frac{a}{2b}.
\]
Assume $k+2 = \frac{a}{2b}$, then $q(L-\frac{a}{2b}\delta) = 2d - \frac{2d^2+5}{2b^2} = -\frac{5}{2b^2}$ is never an integer, hence this case cannot happen and therefore $k+2 < \frac{a}{2b}$. In order to have $\frac{a}{2b} - \frac{2db}{a} = \frac{5}{2ba}>1$ we must have $a =1,2$, one check directly that this not gives solutions to the equation $x^2-4dy^2=5$.
\end{proof}

\subsubsection{An application to higher Gaussian maps of curves}

Having proved that the higher Gaussian maps are surjective for a general polarized K3 surface $(S,L)$ we will prove an analogous result for a smooth member in the linear system $|L|$, which is a canonical curve by adjunction. To do this we will follow a similar strategy to the one employed by Colombo and Frediani in \cite[Main Theorem]{ColFred10}, where the authors prove a vanishing result for linear systems on $S\times S$ to deduce surjectivity of the second Gaussian map.

Let $C$ be a smooth curve in the linear system $|L|$. Consider the following exact sequence induced by restriction:

\begin{equation}\label{eq:p1}
\begin{tikzcd}
0\ar[r] & \text{Sym}^k\Omega^1_S(L)\ar[r] & \text{Sym}^k\Omega^1_S(2L)\ar[r,"p_1"] & \text{Sym}^k\Omega^1_S|_C(2K_C)\ar[r] & 0. 
\end{tikzcd}  
\end{equation}

The conormal bundle of $C$ in $S$ sits in the following exact sequence

\[
\begin{tikzcd}
0\ar[r] & O_C(-K_C) \ar[r] & \Omega^1_S|_C \ar[r] & \O_C(K_C) \ar[r] & 0, 
\end{tikzcd}
\]

taking symmetric powers we give for each $p>0$ the following exact sequence 
\[
\begin{tikzcd}
0\ar[r] & O_C(-pK_C) \ar[r] & \text{Sym}^p\Omega^1_S|_C \ar[r] & \text{Sym}^{p-1}\Omega^1_S|_C(K_C) \ar[r] & 0. 
\end{tikzcd}
\]

Recall that $\Omega^1_S\cong T_S$. We set $p=k$ in the sequence above, dualize and then tensor with $\O_C(2K_C)$ to obtain

\begin{equation}\label{eq:p2}
\begin{tikzcd}
0\ar[r] & \text{Sym}^{k-1}\Omega^1_S|_C(K_C) \ar[r] & \text{Sym}^k\Omega^1_S|_C(2K_C) \ar[r,"p_2"] & \O_C((k+2)K_C) \ar[r] & 0. 
\end{tikzcd}
\end{equation}

We will reduce to analyzing the Picard group of $\P(\Omega^1_S)$, which is generated by the tautological class $\xi$ and the pullback of $\Pic(S)$. Since $\P(\Omega^1_S)$ is a subvariety of $S^{[2]}$, a good knowledge of the nef cone of $S^{[2]}$ gives lower bounds on the \emph{slope} of the nef cone of $\P(\Omega^1_S)$.

\begin{thm}\cite[Proposition 3.2]{GounelasOttem20}\label{thm:gounelasottem}
Let $(S,L)$ be a polarized K3 surface such that $\Pic(S) = \Z L$ and let $L^2 = 2d$. If $d\geq 4(k+2)^2+\frac{5}{4}$, then $L+(k+2)\xi$ is big and nef on $\P(\Omega^1_S)$.
\end{thm}
\begin{cor}\label{cor:vanishingk3general}
Let $(S,L)$ be a polarized K3 surface such that $\Pic(S) = \Z L$ and let $L^2 = 2d$. If $d\geq 4(k+2)^2+\frac{5}{4}$, then $H^1(\P(\Omega^1_S), L+k\xi) = 0$.
\end{cor}
\begin{proof}
With the numerical hypothesis, by Theorem \ref{thm:gounelasottem} we know that $L + (k+2)\xi$ is big and nef on $\P(\Omega^1_S)$. Hence the result follows by Kawamata–Viehweg vanishing theorem, since $H^1(\P(\Omega^1_S), L +k\xi) = H^1(\P(\Omega^1_S), L +(k+2)\xi + K_{\P(\Omega^1_S)})$.
\end{proof}

The bounds given above are sufficient to deduce the surjectivity of higher Gaussian maps for \emph{general} canonical curves of high enough genus.

\begin{thm}\label{thm:higherGauss}
Let $(S,L)$ be a polarized K3 surface such that $\Pic(S) = \Z L$ with $L^2 = 2d$ and $k>1$ an integer. If $d\geq 4(k+2)^2 + \frac{5}{4}$ and $C$ is a smooth hyperplane section of $S$, then $\gamma_C^k$ is surjective.
\end{thm}
\begin{proof}
As in  \cite{ColFred10} consider the commutative diagram induced by naturality of the Gaussian maps:
\[
\begin{tikzcd}
H^0(S\times S,\mathcal{I}_\Delta^k(L\boxtimes L))\ar[r,"\gamma^k_{L}"]\ar[dd,"res"] & H^0(S,\text{Sym}^k\Omega^1_S(2L))\ar[dr,"p_1"]\ar[dd,"res"]\\
 & & H^0((\text{Sym}^k\Omega^1_S)|_C(2K_C))\ar[dl,"p_2"]\\
 H^0(C\times C,\mathcal{I}_\Delta^k(K_C\boxtimes K_C))\ar[r,"\gamma_C^k"] & H^0(C, (k+2)K_C)
\end{tikzcd}
\]
By Theorem \ref{thm:surjectivitygaussmapsK3} the map $\gamma^k_{\O_S(L)}$ is always surjective. In order to verify that $\gamma^k_C$ is surjective we will show that both $p_1$ and $p_2$ are surjective. The morphism $p_1$ is given in the exact sequence \eqref{eq:p1}, and this is surjective if $H^1(S,\text{Sym}^k\Omega^1_S(L)) = 0$. The vanishing of the Leray spectral sequence yields
\[
H^1(S,\text{Sym}^k\Omega^1_S(L)) \cong H^1(\P(\Omega^1_S), L+k\xi) 
\]
and the last group is zero if $d\geq 4(k+2)^2 + \frac{5}{4}$ by Corollary \ref{cor:vanishingk3general}. It remains to show that also $p_2$ is surjective. Since $p_2$ is given via the exact sequence \eqref{eq:p2}, it suffices to show that 
\[
H^1(C,\text{Sym}^{k-1}\Omega^1_S|_C(K_C)) = H^0(C,\text{Sym}^{k-1}T_S|_C)^\vee 
\]
is zero. Restriction to $C$ gives the exact sequence below 
\[
\begin{tikzcd}
0\ar[r] & \text{Sym}^{k-1}T_S(-L) \ar[r] & \text{Sym}^{k-1}T_S\ar[r] & \text{Sym}^{k-1}T_S|_C \ar[r] & 0. 
\end{tikzcd}
\]
A general result due to Kobayashi \cite{Kobayashi80} implies that
$H^0(S,\text{Sym}^{k-1}T_S) = 0$ for all $k>1$, hence $H^0(S,\text{Sym}^{k-1}T_S|_C)$ injects in $H^1(S,\text{Sym}^{k-1}T_S(-L))$, but since
\[
H^1(S,\text{Sym}^{k-1}T_S(-L))^\vee \cong H^1(S,\text{Sym}^{k-1}\Omega^1_S(L)) \cong H^1(\P(\Omega^1_S), L+(k-1)\xi)
\]
we conclude as before that this last group vanishes if $d\geq 4(k+2)^2 +\frac{5}{4}$ by Corollary \ref{cor:vanishingk3general}.
\end{proof}

\subsection{Arbitrary K3 surfaces}

We finish this section by consider the case of arbitrary K3 surfaces. Let $(S,L)$ be a polarized K3 surface. If we let the degree of the divisor $L$ be sufficiently big with respect to $k$ (but still quadratic in $k$), Corollaries \ref{cor:movablehighdegree} and \ref{cor:amplehighdegree} give a necessary and sufficient criterion for a class to be ample, hence we get the surjectivity of the $k$-th higher Gaussian map for hyperplane sections of arbitrary K3 surfaces.

\begin{thm}\label{thm:highergaussmapsgeneral}
Let $(S,L)$ be any polarized K3 surface of degree $2d$ and $k>1$ an integer. If $d>\frac{9}{4}(k+2)^2$, $C$ is a smooth hyperplane section of $|L|$ there are no classes as in Corollaries \ref{cor:movablehighdegree} and \ref{cor:amplehighdegree} for $a=k+2$, then $\gamma^k_C$ is surjective.
\end{thm}
\begin{proof}
With the hypothesis in the Theorem, Corollaries \ref{cor:movablehighdegree} and \ref{cor:amplehighdegree} imply that the class $L-(k+2)\delta$ is ample. Since the restriction of nef divisors is nef we obtain that $(L+(k+2)\delta)|_{\P(\Omega^1_S)}$ is ample, hence the same proof as Theorem \ref{thm:higherGauss} yields the result, noticing that the vanishing of the $k$-th Gaussian map $\gamma_{L}^k$ for $L$ is a direct consequence of Kawamata-Viehweg vanishing and Theorem \ref{thm:vanish}.
\end{proof}

\section{Further remarks}\label{section:furtherremarks}

\subsection{Relation with Bott vanishing}

The proof of Theorem \ref{thm:higherGauss} also produces another criterion for the surjectivity of higher Gaussian maps of a canonical curve lying on a K3 surface depending solely on the vanishing of some cohomology groups attached to twisted symmetric differentials on $S$.

\begin{cor}\label{cor:bottvanishsym}
Let $(S,L)$ be a polarized K3 surface such that $\Pic(S) = \Z L$ and let $C$ be a smooth hyperplane section of $S$. If $H^1(S, \mathrm{Sym}^{k-1}\Omega^1_S(L))$ and $H^1(S, \mathrm{Sym}^{k}\Omega^1_S(L))$ vanish, then $\gamma^k_C$ is surjective.
\end{cor}
\begin{proof}
Going trough the proof of Theorem \ref{thm:higherGauss}, the claimed vanishings imply that the maps $p_1$ and $p_2$ in loc.cit. are surjective. Hence $\gamma^k_C$ is surjective as well.
\end{proof}
\begin{cor}\label{cor:bottvanishingsym2} 
Let $(S,L)$ be a polarized K3 surface of degree $2d\geq 34$ and $\Pic(S) = \Z L$. If $C$ is a smooth hyperplane section of $L$, then $\gamma^2_C$ is surjective if $H^1(S,\mathrm{Sym}^2\Omega^1_S(L))$ vanishes.
\end{cor}
\begin{proof}
Bott vanishing \cite[Theorem 3.2]{Totaro20} imply that $H^1(S,\Omega^1_S(L))$ vanishes. Hence by Corollary \ref{cor:bottvanishsym} the result follows.
\end{proof}

Proving the vanishing stated in the previous corollary will yield surjectivity of the second Gaussian map for general curves of all allowable genus, hence giving an alternative proof of the main result in \cite{CCM11}; obtaining a proof in this way would be interesting because for curves in K3 surfaces the surjectivity of $\gamma^1_C$  never occurs, as mentioned at the beginning of the paper.

On the other hand, in \cite{Totaro20}, Totaro raised the problem of characterizing those polarized K3 surfaces $(S,L)$ for which $H^1(S,\Omega^1_S(L))$ vanishes, such property was called \emph{Bott vanishing} in \emph{loc.cit.} Since symmetric powers of the cotangent bundle appear naturally for higher Gaussian maps, one may expect to have some relation between them as in Corollary \ref{cor:bottvanishingsym2}. Our methods might be pushed forward in order to understand Bott vanishing for symmetric products and to characterize those polarized K3 surfaces $(S,L)$ for which $H^i(S,\mathrm{Sym}^k\Omega^1_S(L))$ vanishes. We plan to come back to this question in the future.

\subsection{Explicit Linear Systems for Hilbert Schemes}

The results of this section were the motivation to describe the Gaussian maps on polarized K3 surfaces. For linear systems on $S^{[2]}$ of the form $L-a\delta$ with $a\leq 2$ we can precisely say whether these are base point free or very ample just in terms of the geometry of $S$. In the following propositions we will give geometric interpretations of the results in Propositions \ref{prop:l-delta} and \ref{prop:lminus2delta}.

\begin{prop}\label{prop:h-delta}
Let $(S,L)$ be a polarized K3 surface. The divisor $L-\delta$ is ample if and only if one of the following occurs:
\begin{itemize}
    \item $L$ is very ample and $(S,L)$ does not contain lines.
    \item $L = 2B$ with $B^2 = 2$ and there are no effective divisors $F$ such that $B\cdot F = 1$ and $F^2=0$. 
\end{itemize}
\end{prop}
\begin{proof}

By Proposition \ref{prop:l-delta} there are no effective divisors $D$ in $S$ such that $D^2 = 0$ and $L\cdot D\leq 2$ if and only if $L-\delta$ is in the interior of the movable cone. By the case $k=1$ in Theorem \ref{thm:knutsen}, we have that this is equivalently to $L$ being very ample or else be of the form $2B$ with $B^2=2$ and such that there are no effective divisors $F$ such that $B\cdot F = 1$ and $F^2=0$.

A line in the projective embedding given by $L$ is a rational curve $D$ in $S$ with $L\cdot D = 1$, and by Riemann-Roch $D^2 = -2$, this is precisely the second case in Proposition \ref{prop:l-delta}.
\end{proof}

Let $(S,L)$ be a polarized K3 surface. Consider the rational map 
\[
\begin{tikzcd}
\varphi_1: S^{[2]}\ar[r, dashed] &  \mathrm{Gr}(2,H^0(S,L))
\end{tikzcd}
\]
given by associating to $Z\in S^{[2]}$ its linear span. If $L$ is base-point free, then by \cite[Proposition 6.9]{AproduNagel10} we have $L-\delta = \varphi_1^*(\O_{\mathrm{Gr}(2,H^0(S,L))}(1)).$
Moreover, if $L$ is very ample, then $\varphi_1$ is a morphism \cite{CatGoe90}. In the second case of Proposition \ref{prop:h-delta} the divisor $2B$ is not very ample and in fact it defines a $2:1$ cover onto the Veronese surface, hence $\varphi_1$ is not defined on a $\P^2$ given by the covering involution.

\begin{cor}
Suppose $L-\delta$ is ample. Then $L-\delta$ is base point free if and only if $L$ is not of the form $L = 2B$, with $B^2 = 2$.
\end{cor}

In a similar way it can be proved an analogue of Proposition \ref{prop:h-delta} when we consider the class $L-2\delta$. This essentially corresponds to $\gamma^0_{L}$.

\begin{prop}\label{prop:h-2delta}
Let $(S,L)$ be a polarized K3 surface. The divisor $L-2\delta$ is ample if and only if one of the following occurs:
\begin{itemize}
    \item $L$ is $3$-very ample.
    \item $L=2B$ with $B^2 = 4$ and there are no effective divisors $E,F$ such that $E^2 = -2$ and $B\cdot E =1$  or $F^2 = 0$ and  $B\cdot F = 2$. 
\end{itemize}
\end{prop}

In both cases we have that $(S,L)$ is generated by quadrics. The linear system $|L-2\delta|$ is given by the morphism
\[
\begin{tikzcd}
\varphi_2: S^{[2]}\ar[r] & \P(H^0(\P^{d+1} , \mathcal{I}_S(2))^\vee)
\end{tikzcd}
\]
that associates to each $Z\in S^{[2]}$ the hyperplane of quadrics containing $S$ and the line spanned by $Z$; since $(S,L)$ is generated by quadrics, we have that $\varphi_2$ is always defined. Moreover it can be shown that whenever $L$ is $3$-very ample we have that $\varphi_2$ is an embedding. In the second case of Proposition \ref{prop:h-2delta} the numerical condition is the same as requiring that $B$ is very ample and $(S,B)$ does not contain lines. 

\begin{cor}
If $L-2\delta$ is ample, then is base point free. It is very ample if and only if it is not of the form $L = 2B$, with $B^2 = 4$.
\end{cor}

Let us suppose that $L-3\delta$ is ample, so in particular $(S,L)$ is generated by quadrics. Our identifications in Section \ref{section:gaussianandhilbertscheme} yield that the space of global sections of the divisor $L-3\delta$ is the kernel of the Wahl map:
\[
H^0(S^{[2]}, L-3\delta) = \ker ( \bigwedge^2 H^0(S,L)\to H^0(S, \Omega^1_S(2L))
\]
One can show that $H^0(S^{[2]}, L-3\delta) \cong H^1(\P^{d+1}, \mathcal{I}_S^2(2))$.

\begin{qsn}
Is there an analogous geometric description of $|L-3\delta|$? Is it true that $|L-3\delta|$ is always very ample?
\end{qsn}

As a final remark, it would be interesting to investigate the relation between the so-called higher fundamental forms described in \cite{GriffithsHarris79} and global sections of divisors of the form $L-a\delta$, whenever this class is ample. 

\appendix

\section{Equivariant Sheaves}\label{section:equivariantsheaves}

In this appendix we will collect the basic theory concerning $G$-equivariant sheaves. For more details, the reader can consult \cite{BKR01}.

Let $X$ be a projective variety together with an action of a finite group $G$. Denote by $e\in G$ the identity element. A $G$-equivariant sheaf will be a sheaf such that the action of $G$ on $X$ \emph{lifts}. We give the formal definition below.

\begin{defn}
A $G$-equivariant sheaf on $X$ is a coherent sheaf $\F$ on $X$ together with isomorphisms $\lambda_g:\F\to g^*\F$ for every $g\in G$ such that $\lambda_{e} = \text{Id}$ and  $\lambda_{hg} = g^*(\lambda_h)\circ \lambda_g$ for every $g,h\in G$. 
\end{defn}

The morphisms $\lambda_g$ are also called \emph{linearizations}.

\begin{exmp}
Let $\rho: G\to \text{GL}(V)$ be a representation of $G$. For every $G$-equivariant sheaf $\F$ we define $\F\otimes \rho$ as the sheaf $\F\otimes_{\C} V$ together with linearizations $s\otimes v\mapsto \lambda_g(s)\otimes \rho(g)(v)$. Notice that the sheaf $\F\otimes_\C V$ is still G-equivariant.
\end{exmp}

If $G$ acts trivially on $X$, then a $G$-equivariant sheaf is a sheaf with an action of $G$. In particular, if $\F$ is a $G$-equivariant sheaf on $X$ and $\pi: X\to X/G$ the quotient morphism, then $\pi_*\F$ is naturally a $G$-equivariant sheaf for the trivial action of $G$ in $X/G$. 

\begin{defn}
Let $\F$ be a $G$-equivariant sheaf on $X$ and $\pi:X\to X/G$ the quotient morphism. The \emph{sheaf of invariants} of $\F$, denoted by $\F^G$ is defined for every open $U\subset X/G$ as
\[
\F^G(U) = H^0(U,\pi_*\F)^G = \{ s\in H^0(U,\pi_*\F) : g\cdot s = s \}
\]
\end{defn}

\begin{rem}\label{rem:sheafofinvariants}
Notice that the sheaf of invariants can only be defined as a sheaf over the quotient $X/G$ and not over $X$. However for every sheaf $\F'$ on $X/G$, the sheaf $\pi^*\F'$ is a $G$-equivariant sheaf on $X$ and we have that \cite[Section 4.2]{BKR01} 
\[
(\pi_*(\pi^*\F'))^G\cong \F'
\]
\end{rem}

\begin{lem}\label{lem:equivsheafsplits}
Let $\F$ be a $G$-equivariant sheaf over $X$. If $G$ acts trivially over $X$, then $\F$ splits
\begin{equation}\label{eq:equivsheafsplits}
\F \cong \bigoplus_\rho \F_\rho\otimes\rho,
\end{equation}
where $\rho$ are the ireducible representations of $G$. This decomposition is unique up to ordering of the factors. In particular, the subsheaf of invariants $\F^G$ corresponds to the trivial part in the decomposition.
\end{lem}
\begin{proof}
This is a consequence of Schur's Lemma (see \cite{BKR01}).
\end{proof}

\bibliography{bib.bib}{}

\begin{thebibliography}{CHM88}

\bibitem[ABS17]{ABS17}
Enrico Arbarello, Andrea Bruno, and Edoardo Sernesi.
\newblock On hyperplane sections of {K}3 surfaces.
\newblock {\em Algebr. Geom.}, 4(5):562--596, 2017.

\bibitem[AN10]{AproduNagel10}
Marian Aprodu and Jan Nagel.
\newblock {\em Koszul cohomology and algebraic geometry}, volume~52 of {\em
  University Lecture Series}.
\newblock American Mathematical Society, Providence, RI, 2010.

\bibitem[Bea83]{Beau83}
Arnaud Beauville.
\newblock Vari\'et\'es {K}\"ahleriennes dont la premi\`ere classe de {C}hern
  est nulle.
\newblock {\em J. Differential Geom.}, 18(4):755--782 (1984), 1983.

\bibitem[BKR01]{BKR01}
Tom Bridgeland, Alastair King, and Miles Reid.
\newblock The {M}c{K}ay correspondence as an equivalence of derived categories.
\newblock {\em J. Amer. Math. Soc.}, 14(3):535--554, 2001.

\bibitem[BM14]{BMMMP14}
Arend Bayer and Emanuele Macr\`\i.
\newblock M{MP} for moduli of sheaves on {K}3s via wall-crossing: nef and
  movable cones, {L}agrangian fibrations.
\newblock {\em Invent. Math.}, 198(3):505--590, 2014.

\bibitem[CCM11]{CCM11}
Alberto Calabri, Ciro Ciliberto, and Rick Miranda.
\newblock The rank of the second {G}aussian map for general curves.
\newblock {\em Michigan Math. J.}, 60(3):545--559, 2011.

\bibitem[CF09]{CF09}
Elisabetta Colombo and Paola Frediani.
\newblock Some results on the second {G} aussian map for curves.
\newblock {\em Michigan Math. J.}, 58(3):745--758, 2009.

\bibitem[CF10]{ColFred10}
Elisabetta Colombo and Paola Frediani.
\newblock On the second {G}aussian map for curves on a {$K3$} surface.
\newblock {\em Nagoya Math. J.}, 199:123--136, 2010.

\bibitem[CGt90]{CatGoe90}
Fabrizio Catanese and Lothar G\oe~ttsche.
\newblock {$d$}-very-ample line bundles and embeddings of {H}ilbert schemes of
  {$0$}-cycles.
\newblock {\em Manuscripta Math.}, 68(3):337--341, 1990.

\bibitem[CHM88]{CHM88}
Ciro Ciliberto, Joe Harris, and Rick Miranda.
\newblock On the surjectivity of the {W}ahl map.
\newblock {\em Duke Math. J.}, 57(3):829--858, 1988.

\bibitem[CPT01]{CPT01}
Elisabetta Colombo, Gian~Pietro Pirola, and Alfonso Tortora.
\newblock Hodge-{G}aussian maps.
\newblock {\em Ann. Scuola Norm. Sup. Pisa Cl. Sci. (4)}, 30(1):125--146, 2001.

\bibitem[Fog73]{Fogarty73}
J.~Fogarty.
\newblock Algebraic families on an algebraic surface. {II}. {T}he {P}icard
  scheme of the punctual {H}ilbert scheme.
\newblock {\em Amer. J. Math.}, 95:660--687, 1973.

\bibitem[GH79]{GriffithsHarris79}
Phillip Griffiths and Joseph Harris.
\newblock Algebraic geometry and local differential geometry.
\newblock {\em Ann. Sci. \'{E}cole Norm. Sup. (4)}, 12(3):355--452, 1979.

\bibitem[GO20]{GounelasOttem20}
Frank Gounelas and John~Christian Ottem.
\newblock Remarks on the positivity of the cotangent bundle of a {K}3 surface.
\newblock {\em \'{E}pijournal G\'{e}om. Alg\'{e}brique}, 4:Art. 8, 16, 2020.

\bibitem[Har77]{Hart77}
Robin Hartshorne.
\newblock {\em Algebraic geometry}.
\newblock Springer-Verlag, New York-Heidelberg, 1977.
\newblock Graduate Texts in Mathematics, No. 52.

\bibitem[Knu01]{Knutsen01}
Andreas~Leopold Knutsen.
\newblock On {$k$}th-order embeddings of {$K3$} surfaces and {E}nriques
  surfaces.
\newblock {\em Manuscripta Math.}, 104(2):211--237, 2001.

\bibitem[Kob80]{Kobayashi80}
Shoshichi Kobayashi.
\newblock First {C}hern class and holomorphic tensor fields.
\newblock {\em Nagoya Math. J.}, 77:5--11, 1980.

\bibitem[Laz04]{PAG1}
Robert Lazarsfeld.
\newblock {\em Positivity in algebraic geometry. {I}}, volume~48 of {\em
  Ergebnisse der Mathematik und ihrer Grenzgebiete. 3. Folge. A Series of
  Modern Surveys in Mathematics [Results in Mathematics and Related Areas. 3rd
  Series. A Series of Modern Surveys in Mathematics]}.
\newblock Springer-Verlag, Berlin, 2004.
\newblock Classical setting: line bundles and linear series.

\bibitem[LQZ03]{LQZ03}
Wei-Ping Li, Zhenbo Qin, and Qi~Zhang.
\newblock Curves in the {H}ilbert schemes of points on surfaces.
\newblock In {\em Vector bundles and representation theory ({C}olumbia, {MO},
  2002)}, volume 322 of {\em Contemp. Math.}, pages 89--96. Amer. Math. Soc.,
  Providence, RI, 2003.

\bibitem[Mat05]{Mat05}
Daisuke Matsushita.
\newblock Higher direct images of dualizing sheaves of {L}agrangian fibrations.
\newblock {\em Amer. J. Math.}, 127(2):243--259, 2005.

\bibitem[Sca20]{Scala20}
Luca Scala.
\newblock Notes on diagonals of the product and symmetric variety of a surface.
\newblock {\em J. Pure Appl. Algebra}, 224(8):106352, 48, 2020.

\bibitem[SD74]{SD1974}
B.~Saint-Donat.
\newblock Projective models of {$K-3$} surfaces.
\newblock {\em Amer. J. Math.}, 96:602--639, 1974.

\bibitem[Tot20]{Totaro20}
Burt Totaro.
\newblock Bott vanishing for algebraic surfaces.
\newblock {\em Trans. Amer. Math. Soc.}, 373(5):3609--3626, 2020.

\bibitem[Ver07]{Verbitsky07}
Misha Verbitsky.
\newblock Quaternionic {D}olbeault complex and vanishing theorems on
  hyperk\"{a}hler manifolds.
\newblock {\em Compos. Math.}, 143(6):1576--1592, 2007.

\bibitem[Voi92]{Voisin92}
Claire Voisin.
\newblock Sur l'application de {W}ahl des courbes satisfaisant la condition de
  {B}rill-{N}oether-{P}etri.
\newblock {\em Acta Math.}, 168(3-4):249--272, 1992.

\bibitem[Wah87]{Wahl87}
Jonathan~M. Wahl.
\newblock The {J}acobian algebra of a graded {G}orenstein singularity.
\newblock {\em Duke Math. J.}, 55(4):843--871, 1987.

\bibitem[Wah90]{Wahl90}
Jonathan Wahl.
\newblock Gaussian maps on algebraic curves.
\newblock {\em J. Differential Geom.}, 32(1):77--98, 1990.

\bibitem[Wah92]{Wahl92}
Jonathan Wahl.
\newblock Introduction to {G}aussian maps on an algebraic curve.
\newblock In {\em Complex projective geometry ({T}rieste, 1989/{B}ergen,
  1989)}, volume 179 of {\em London Math. Soc. Lecture Note Ser.}, pages
  304--323. Cambridge Univ. Press, Cambridge, 1992.

\end{thebibliography}
\bibliographystyle{alpha}

\end{document}